\RequirePackage{rotating}
\documentclass[12pt]{amsart}
\usepackage[graphicx]{realboxes}
\usepackage{float}
\restylefloat{table}
\usepackage{placeins}

\usepackage{amsmath}
\usepackage{amssymb}
\usepackage{epsfig}
\usepackage{wasysym}
\usepackage{graphicx}
\usepackage{tikz}
\usepackage{tikz-cd}
\usepackage{bm}
\usepackage{xcolor}
\usepackage{listings}
\numberwithin{equation}{section}
\usepackage{extpfeil}
\usepackage{array}
\usepackage{adjustbox}
\usepackage{longtable}
\usepackage{makecell}
\usepackage[all]{xy}
\usepackage{longtable}
\usepackage{rotating}
\input xy
\xyoption{all}
\usepackage{multirow}

\usetikzlibrary{positioning}
\usepackage[colorlinks=false,urlbordercolor=white]{hyperref}
  
\tikzset{sgplattice/.style={inner sep=1pt,norm/.style={red!50!blue},char/.style={blue!50!black},
  lin/.style={black!50}},cnj/.style={black!50,yshift=-2.5pt,left=-1pt of #1,scale=0.5,fill=white}}

\DeclareFontFamily{U}{mathb}{\hyphenchar\font45}
\DeclareFontShape{U}{mathb}{m}{n}{
      <5> <6> <7> <8> <9> <10> gen * mathb
      <10.95> mathb10 <12> <14.4> <17.28> <20.74> <24.88> mathb12
      }{}
\DeclareSymbolFont{mathb}{U}{mathb}{m}{n}
\DeclareMathSymbol{\righttoleftarrow}{3}{mathb}{"FD}

\calclayout
\allowdisplaybreaks[3]

\theoremstyle{plain}
\newtheorem{prop}{Proposition}[section]

\newtheorem{theo}[prop]{Theorem}
\newtheorem{coro}[prop]{Corollary}

\theoremstyle{definition}

\newtheorem{rema}[prop]{Remark}

\newcommand{\actsfromleft}{\mathrel{\reflectbox{$\righttoleftarrow$}}}

\def\cE{{\mathcal E}}

\def\cO{{\mathcal O}}

\def\sA{{\mathsf A}}

\def\fA{{\mathfrak A}}

\def\fD{{\mathfrak D}}

\def\fS{{\mathfrak S}}

\def\fS{{\mathfrak S}}

\def\bA{{\mathbb A}}
\def\bG{{\mathbb G}}
\def\bP{{\mathbb P}}

\def\Aut{\mathrm{Aut}}

\def\GL{\mathsf{GL}}

\def\PGL{\mathsf{PGL}}

\def\lim{\mathrm{lim}}

\def\Sing{\mathrm{Sing}}

\makeatother
\makeatletter

\begin{document}
\title[Equivariant geometry of cubic threefolds]{Equivariant geometry of cubic threefolds with non-isolated singularities}

\author[I. Cheltsov]{Ivan Cheltsov}
\address{Department of Mathematics, University of Edinburgh, UK}

\email{I.Cheltsov@ed.ac.uk}

\author[L. Marquand]{Lisa Marquand}
\address{
  Courant Institute,
  251 Mercer Street,
  New York, NY 10012, USA
}

\email{lisa.marquand@nyu.edu}

\author[Y. Tschinkel]{Yuri Tschinkel}
\address{
  Courant Institute,
  251 Mercer Street,
  New York, NY 10012, USA
}

\email{tschinkel@cims.nyu.edu}

\address{Simons Foundation\\
160 Fifth Avenue\\
New York, NY 10010\\
USA}

\author[Zh. Zhang]{Zhijia Zhang}

\address{
Courant Institute,
  251 Mercer Street,
  New York, NY 10012, USA
}

\email{zz1753@nyu.edu}

\date{\today}
%\dedicatory{To James McKernan on the occasion of his 60th birthday.}

\begin{abstract}
We study linearizability of actions of finite groups on cubic threefolds with non-isolated singularities.  
\end{abstract}

\maketitle

\section{Introduction}
\label{sect:intro}

Let $k$ be an algebraically closed field of characteristic zero, $X$ a smooth projective variety over $k$ of dimension $n$ and $G\subseteq \Aut(X)$ a subgroup of automorphisms. The $G$-action on $X$ is {\em linearizable} if it is equivariantly birational to a linear $G$-action on $\bP^n$, and  {\em unirational} if $X$ is $G$-equivariantly dominated by the projectivization of a linear representation of $G$.

In \cite{CTZ-burk}, \cite{CTZ-cubic}, \cite{CMTZ} and \cite{CTZ-uni}, we have addressed the problems of unirationality and linearizability of actions of finite groups on cubic threefolds with {\em isolated} singularities. 
In this paper, we extend the study of birational properties of generically free regular actions of finite groups on singular cubic threefolds to those with {\em non-isolated} singularities. 

Detailed knowledge of degenerations of cubic threefolds, together with their automorphisms, plays an important role in moduli theory, see, e.g., \cite{All}, \cite{CGH}. Indeed, the classification of isolated singularities in \cite{viktorova} was one of the motivations of our work. Another source of inspiration comes from arithmetic and birational geometry over nonclosed fields, as in \cite{CTS} or \cite{Colliot-P}.

We proceed to describe our strategy: we start by identifying all possibilities for actions
and the corresponding normal forms. Roughly, cubic threefolds with non-isolated singularities are of four types, according to the geometry of the singular locus:
line,  conic, plane, twisted quartic.
%(there is a unique such cubic called the chordal cubic). 
More precisely, 
denote by $\Sing(X)$ the singular locus of $X$. Suppose that $\dim(\Sing(X))\geq 1$, and $X$ is not a cone.   
We observe that the secant variety of $\Sing(X)$ is either $X$ or has dimension $\le 2$; otherwise, $X$ is reducible. 
Building on this, the possibilities of $\Sing(X)$ have been classified in \cite[Proposition 4.2]{yokocub}:
\begin{itemize}
    \item When $\mathrm{dim}(\Sing(X))\geq 2$, then $\Sing(X)$ is a plane.
    \item When $\mathrm{dim}(\Sing(X))=1$ and $\Sing(X)$ contains a plane curve, then the union of one-dimensional components of $\Sing(X)$ is either a line, a smooth conic, or two lines intersecting at a point.
    \item When $\mathrm{dim}(\Sing(X))=1$ and $\Sing(X)$ contains a curve not contained in a plane, then $\Sing(X)$
is a smooth rational normal quartic curve and there is a unique such $X$, known as the {\em chordal cubic}, given by 
$$
    x_1x_4^2 + x_2^2x_5 - x_1x_3x_5 - 2x_2x_3x_4 + x_3^3=0.
$$
\end{itemize}

We focus on actions not fixing singular points on $X$, since otherwise, the actions are linearizable.  The resulting cases are analyzed using a variety of tools from birational geometry, including the connection between equivariant geometry and geometry over nonclosed fields in \cite{DR}. We also restrict to finite group actions.
Our principal results are:
\begin{itemize} 
    \item All actions on cubics singular along a line are linearizable, by  Theorem~\ref{thm:line}. 
    \item All actions on the chordal cubic are linearizable, by Theorem~\ref{theo:chordalmain}.
    \item All actions on cubics singular along a plane are linearizable, by Theorem~\ref{thm:plane}.
    \item  Most actions on cubics singular along a conic are not linearizable, by Theorem~\ref{thm:12}. The linearizability problem of some actions in this case remains open; 
   see Section~\ref{sect:conic} for more details. 
\item All actions on cubic threefolds with non-isolated singularities are unirational, by Theorem~\ref{thm:lineuni1}. 
\end{itemize}

One of the corollaries of classifications in \cite{All}, \cite{viktorova}, \cite{yokocub}, \cite{LiuXu} is the following: if $X\subset \bP^4$ is a K-unstable rational cubic threefold with isolated singularities and $G\subseteq \Aut(X)$ is a finite subgroup, then the $G$-action on $X$ is linearizable. This is no longer true for cubic threefolds with non-isolated singularities -- there are non-linearizable $G$-actions on K-unstable cubic 3-folds singular along a conic; this is also the most interesting case in arithmetic considerations in \cite{Colliot-P}.

\

\noindent
{\bf Acknowledgments:} 
The first author was supported by Simons Collaboration grant \emph{Moduli of varieties}. He also thanks CIRM, Luminy, for its hospitality and for
providing a perfect work environment. The third author was partially supported by NSF grant 2301983.

\section{Tools from birational geometry}
\label{sect:gen}

We recall the basic terminology: a $G$-action on $\bP^n$ is called {\em linear} if it arises from projectivization $\bP(V)$ of a $G$-representation $V$. We {\em do not} assume that the action on $\bP(V)$ is generically free. 
We will use the following general results. 

\begin{prop}
\label{prop:25}
     Let $X=\bP_{B}(\cE)$ be the projectivization of a vector bundle $\cE$ of rank $n+1$ over a smooth projective irreducible variety $B$, and $\pi:X\to B$ the associated $\bP^n$-bundle.
    Assume that $X$ carries a regular action of a finite group $G$ such that $\cE$ is $G$-linearized, 
    and that the induced action on $B$ is unirational. Then the action on $X$ is $G$-unirational. 
    
    Moreover, if the $G$-action on $B$ is linearizable and generically free then $X$ is linearizable.   
\end{prop}

\begin{proof}
By assumption, there exists a $G$-representation $V$ such that $G$ acts generically freely on $\bP(V)$, 
which in turn admits a dominant $G$-equivariant rational map to $B$. Equivariantly resolving indeterminacy of this map $\widetilde{\bP(V)}\to\bP(V)$, we obtain
the following $G$-equivariant diagram:
$$
\xymatrix{
\cE \ar[d]        & \ar@{->}[l] \tilde{\cE}  \ar[d]    \\ 
B      &  \ar@{->}[l] \widetilde{\bP(V)} 
}
$$
where $\tilde{\cE}$ is the pullback of $\cE$ to $\widetilde{\bP(V)}$.
By assumption, the $G$-action on $\bP_B(\cE)$ lifts to $\cE$. It follows that $\tilde{\cE}$ is 
a $G$-linearized vector bundle over $\widetilde{\bP(V)}$. By the no-name lemma, see, e.g., \cite[Theorem 1]{Kang}, $\tilde{\cE}$ is $G$-equivariantly birational to $\widetilde{\bP(V)}\times \bA^{n+1}$,
with trivial action on the second factor, 
as the $G$-action on $\widetilde{\bP(V)}$ is generically free. It follows that $X=\bP_B(\cE)$ is dominated by $\bP(W)$, for some $G$-representation $W$, thus is $G$-unirational. 

The second claim is a projective version of the no-name lemma, see, e.g., \cite[Theorem 1']{Kang}.
\end{proof}

\begin{rema} 
This proposition applies in particular to actions on products of projective spaces, e.g., $\bP^1\times \bP^2$ in Section~\ref{sect:line}, or to a nontrivial $\bP^1$-bundle over $\bP^2$, in Section~\ref{sect:chordal}. 
In general, e.g, when the $G$-action on the base $B$ is not generically free, or does not lift to $\cE$, the linearization problem remains a challenge. The linearizability of actions on $(\bP^1)^2$ was recently settled in \cite{pinardin}; the case of $(\bP^1)^3$ is open.
\end{rema}

The following proposition is a version of \cite[Propositions 3.1 and 3.5]{CTZ-uni}.

\begin{prop}
\label{prop:unirat}
Let $k$ be an algebraically closed field of characteristic 0. Let $X\subset \bP^n$, $n\ge 3$, be an irreducible cubic hypersurface over $k$ which is not a cone, and $G\subseteq \Aut(X)$ a finite subgroup, acting linearly on $\bP^n$. Assume that $X$ contains a $G$-invariant subvariety  $S$, which is $G$-unirational. Then $X$ is $G$-unirational.
\end{prop}

\begin{proof}
By \cite[Theorem 1.1]{DR}, $G$-unirationality of $X$ is equivalent to the following property: for every field  $K/k$ and every $G$-torsor $T$ over $K$, the twist ${}^TX_K$ of $X$ over $K$ via $T$ is $K$-unirational. Our assumption implies that all such twists are cubic hypersurfaces in $\bP^n_K$,
see \cite[Lemma 10.1]{DR}. Every such twist contains a twisted form of $S$, defined over $K$. By assumption and \cite[Theorem 1.1]{DR}, 
this twisted form of $S$ is unirational over $K$ and, in particular, 
has $K$-rational points. Then ${}^TX_K$ also has $K$-rational points and is $K$-unirational by \cite{kollarcubic}.
\end{proof}

\begin{rema}
This proposition applies in particular when $X$ contains a $G$-invariant linear subspace, e.g., a point, a line, or a plane; we use it in the proof of Theorem~\ref{thm:lineuni1}.      
\end{rema}

\section{Line}
\label{sect:line}

Let $X\subset \bP^4$ be an irreducible cubic threefold, singular along a line $\mathfrak l$.
If $\Sing(X)$ contains another positive-dimensional component, then the action of $\Aut(X)$ on $X$ is linearizable; indeed, the other component must be a line intersecting $\mathfrak l$ in a distinguished singular point.  
Hence we assume that $\mathfrak l$ is the only positive-dimensional component of $\Sing(X)$.
The main theorem of this section is the following.
\begin{theo}
\label{thm:line}
Let $X$ be a cubic threefold, singular along a line, %which is the unique positive-dimensional component of $\Sing(X)$, 
and $G\subseteq \Aut(X)$ a finite subgroup. 
Then the $G$-action on $X$ is linearizable. 
\end{theo}

A normal form for $X$ is given by
\begin{equation} 
\label{eqn:linesing}
x_1q_1(x_3,x_4,x_5)+x_2q_2(x_3,x_4,x_5)+c(x_3,x_4,x_5)=0,
\end{equation}
where $\mathfrak{l}=\{x_3=x_4=x_5=0\}$, $q_1$ and $q_2$ are quadratic forms, and $c$ is a cubic form, see \cite[Proposition 4.2]{yokocub}. We have a natural identification $\mathfrak{l}=\bP^1_{x_1,x_2}$.
Let $\beta:\tilde{\bP}\to \bP^4$ be the blowup of $\mathfrak{l}$. This yields the commutative diagram:
$$
\xymatrix{
& \tilde{\bP}\ar[ld]_{\beta}\ar[rd]^{\varpi} &\\
\bP^4\ar@{-->}[rr] && \bP^2_{x_3,x_4,x_5}}
$$
where $\varpi$ is a $\bP^2$-bundle, 
and the dashed arrow is the projection from $\mathfrak{l}$.

Let $E$ be the $\beta$-exceptional divisor. 
Then $\beta$ and $\varpi$ induce a natural isomorphism 
$$
E\simeq \mathfrak{l}\times\bP^2_{x_3,x_4,x_5}=\bP^1_{x_1,x_2}\times\bP^2_{x_3,x_4,x_5}.
$$
Let $\tilde{X}$ be the strict transform of $X$ on the fourfold $\tilde{\bP}$.
We have the induced $\Aut(X)$-equivariant commutative diagram: 
$$
\xymatrix{
& \tilde{X}\ar[ld]_{\beta\vert_{\tilde{X}}}\ar[rd]^{\pi} &\\
X\ar@{-->}[rr] && \bP^2_{x_3,x_4,x_5}}
$$
where $\pi=\varpi\vert_{\tilde{X}}$ is a morphism such that its fibers are isomorphic either to $\bP^1$ or $\bP^2$. 
Put $S:=\tilde{X}\vert_{E}$. Then $S$ is a divisor of  bi-degree $(1,2)$ in $E\simeq \bP^1_{x_1,x_2}\times\bP^2_{x_3,x_4,x_5}$ that is given by
$$
x_1q_1(x_3,x_4,x_5)+x_2q_2(x_3,x_4,x_5)=0.
$$
For general $q_1$ and $q_2$, $S$ is a smooth del Pezzo surface of degree $5$, the natural projection $S\to\bP^2_{x_3,x_4,x_5}$ is a blowup of $4$ points in general position, and the natural projection $S\to\bP^1_{x_1,x_2}$ is a conic bundle with $3$ singular fibers. If $q_1$ and $q_2$ are special, $S$ may be singular.

First, assume that $S$ is smooth. In this case, $X$ can be given by 
\begin{multline}
  \label{eqn:xsm}
x_1(x_3^2+\zeta_3x_4^2+\zeta_3^2x_5^2)+x_2(x_3^2+\zeta_3^2x_4^2+\zeta_3x_5^2)+\\
+a_2(x_3^2x_4+x_3x_4^2+x_3^2x_5+x_4^2x_5-x_3^3-x_4^3)+\\
+a_1x_3x_4x_5+a_2x_3x_5^2+a_3x_4x_5^2+(a_4-2a_2)x_5^3=0,
\end{multline}
for some $a_1,a_2,a_3,a_4$, the morphism $\varpi\vert_{S}:S\to\bP^2_{x_3,x_4,x_5}$ is a blowup of the points 
\begin{equation}
\label{eq:4pts}
    [1:1:1], \quad [1:1:-1],\quad [1:-1:1],\quad [-1:1:1],
\end{equation}
and the singular fibers of the conic bundle $\beta\vert_{S}\colon S\to\mathfrak{l}$ lie over the points 
\begin{equation}
    \label{eqn:points}
[1:-1:0:0:0], \quad [1:-\zeta_3:0:0:0], \quad [1:-\zeta_3^2:0:0:0].    
\end{equation}

Let $G\subseteq \Aut(X)$ be a finite subgroup. The set of  points 
in \eqref{eqn:points} must be a $G$-orbit of length $3$, unless one of them is fixed by $G$. 
Note that the image of $G$ in $\mathrm{Aut}(\mathfrak{l})\simeq\PGL_2(k)$ is contained in the subgroup isomorphic to $\fS_3$, generated by 
\begin{align}\label{eqn:S3line}
    (x_1,x_2)\mapsto(x_2,x_1)
\end{align}
and 
$$
(x_1,x_2)\mapsto(\zeta_3x_1,\zeta_3^2x_2).
$$
Since this subgroup has an orbit of length $2$ consisting of the points $[1:0]$ and $[0:1]$, we see that the points
$$
[1:0:0:0:0], \quad  [0:1:0:0:0]
$$
form a $G$-orbit in $\mathfrak{l}$ of length $2$, unless both are fixed by $G$. This allows us to classify all automorphism groups: 

\begin{prop}
\label{prop:1-line}
Let $X$ be a cubic threefold singular along a line and such that the associated degree-5 del Pezzo surface is smooth. Assume that $\Aut(X)$ does not fix a singular point of $X$. Then, up to isomorphism, $X$ is given by \eqref{eqn:xsm} and one of the following holds: 
    \begin{enumerate}
        \item $a_1\in k$, $a_2=a_3=a_4=1$, $\Aut(X)=\fS_3$ is generated by 
        \begin{align*}
           \sigma_1&: (\mathbf x)\mapsto (x_2,x_1,x_3,x_5,x_4),\\
               \sigma_2&: (\mathbf x)\mapsto (\zeta_3x_1,\zeta_3^2x_2,x_4,x_5,x_3).
        \end{align*}
             \item $a_1=1$, $a_2=a_3=a_4=0$, $\Aut(X)=\fS_4$ is generated by $\sigma_1,\sigma_2$ and
       $$
           \iota_1=\mathrm{diag}(1,1,1,-1,-1).
        $$
     \item $a_1=a_2=a_3=a_4=0$, $\Aut(X)=\bG_m\times\fS_4$ is generated by $\sigma_1,\sigma_2,\iota_1$ and 
     $$
     \tau_a=\mathrm{diag}(1,1,a,a,a),\quad a\in k^\times.
     $$
    \end{enumerate}
   
\end{prop}

\begin{proof}
By assumption, the induced $G$-action on $\mathfrak l$ is the $\fS_3$-action generated by \eqref{eqn:S3line}. This action comes from $\Aut(X)$ 
only when 
$$
a_2=a_3=a_4.
$$
There is an exact sequence 
$$
0\to H\to G\to\fS_3\to 0,
$$
where $H$ is the generic stabilizer of $\mathfrak l$, consisting of elements of the form 
$$
\begin{pmatrix}
    1&0&0&0&0\\
    0&1&0&0&0\\
    b_1&c_1&d_1&e_1&f_1\\
     b_2&c_2&d_2&e_2&f_2\\
      b_3&c_3&d_3&e_3&f_3
\end{pmatrix}.
$$
The condition that such elements leave invariant \eqref{eqn:xsm} gives a system of equations. Solving for the parameters $b_i,c_i,d_i,e_i,f_i$ leads to the three cases in the assertion. 
\end{proof}

%Note that the fibers of $\pi$ are mapped by $\beta$ either to lines intersecting $\mathfrak{l}$ or to planes containing $\mathfrak{l}$. 
%Analyzing the projection $X\dasharrow\bP^2_{x_3,x_4,x_5}$ in more detail, we see that $\pi$ is a $\mathbb{P}^1$-bundle if and only if 
%$$
%(a_1+a_2+a_3+a_4)(a_1-a_2+a_3-a_4)(a_1+a_2-a_3-a_4)(a_1-a_2-a_3+a_4)\ne 0. 
%$$
%In fact, all fibers of $\pi$ over points different from \eqref{eq:4pts} are isomorphic to $\bP^1$, and 
%\begin{itemize}
%\item $\pi^{-1}([1:1:1])\simeq\bP^2$ if and only if $a_1+a_2+a_3+a_4=0$,
%\item $\pi^{-1}([1:1:-1])\simeq\bP^2$ if and only if $a_1-a_2-a_3+a_4=0$,
%\item $\pi^{-1}([1:-1:1])\simeq\bP^2$ if and only if $a_1-a_2+a_3-a_4=0$,
%\item $\pi^{-1}([-1:1:1])\simeq\bP^2$ if and only if $a_1+a_2-a_3-a_4=0$.
%\end{itemize}
%\footnote{I CHECKED FEW CASES. IF ONE OF THESE EQUATIONS HOLDS THEN THE PLANE CONTAINS EXTRA SINGULAR POINT, SO TWO OF THEM CAN HOLD or FOUR. If we have 3 planes, then the line over 3rd point would intersect $\mathfrak{l}$ in $G$-invariant points. THIS MUST BE DOUBLE CHECKED.}

We turn to the case when $S$ is singular.

\begin{prop}
\label{prop:s}
Let $X$ be a cubic threefold singular along a line and such that the associated degree-5 del Pezzo surface $S$ is singular. Assume that  
$\Aut(X)$ does not fix a singular point on $X$. 
Then, up to isomorphism, one of the following holds: 
\begin{itemize}
\item[(4)] $X=\{x_1x_3^2+x_2x_4^2+x_5^3=0\}$, and $\Aut(X)=\bG_m^2\rtimes C_2$, generated by 
$$
\mathrm{diag}(t_1^{-2},1,t_1,1,1),\quad \mathrm{diag}(1,t_2^{-2},1,t_2,1),
$$
$$
(\mathbf{x})\mapsto (x_2,x_1,x_4,x_3,x_5),\quad t_1,t_2\in k^\times.
$$
\item[(5)] $X:=\{ x_1x_3^2+x_2x_4^2+x_3x_4x_5+x_5^3=0\}$, and $\Aut(X)=\bG_m\rtimes C_2$, generated by 
$$
\mathrm{diag}(t^{-2},t^{2},t,t^{-1},1),
\quad
(\mathbf{x})\mapsto (x_2,x_1,x_4,x_3,x_5),\quad t\in k^\times.
$$

%\item $x_1x_3^2+x_2x_4^2+x_3x_5^2+x_4x_5^2=0$.
\end{itemize}
\end{prop} 

\begin{proof}
We start by observing that $q_1$ and $q_2$ from \eqref{eqn:linesing} are linearly independent, since $S$ is not a cone.
Now, we suppose that $G$ does not fix points in $\mathfrak{l}$. In particular, $G$ does not fix points in $S$.

If $S$ has Du Val singularities, all possibilities for $S$ are described in \cite[Proposition 8.5]{coray}. 
In particular, since $\Aut(X)$ does not fix points in $S$, we see that $S$ does not have a distinguished singular point, 
which leaves only one possibility for the singular locus of $S$ -- it consists of two singular point of type $\mathsf A_1$. 
But in this case, the conic bundle $S\to\mathfrak l$ has exactly two singular fibers, one of them is not reduced, and the other is reduced and contained in the smooth locus of $S$, so they cannot be swapped by the action of $G$, which implies that $G$ fixes two points in $\mathfrak{l}$.

Hence, the singularities of $S$ are not Du Val. 
Thus, up to a change of coordinates, we may assume that $q_1$ and $q_2$ do not depend on $x_5$.
Then, up to a change of variables, we may assume that  
$$
q_1=x_3^2, \quad q_2=x_4^2. 
$$
Now, the equation~\eqref{eqn:linesing} takes the form
$$
x_1x_3^2+x_2x_4^2+c(x_3,x_4,x_5)=0.
$$
After a linear change of $x_1,x_2$, we may assume that $c$ does not contain monomials divisible by $x_3^2$ and $x_4^2$,
so that
$$
c(x_3,x_4,x_5)=c_1x_5^3+x_5^2(c_2x_3+c_3x_4)+c_4x_3x_4x_5.
$$
\begin{itemize} 
\item 
If $c_1\ne 0$, we can change  $x_1$, $x_2$ and $x_5$ to get 
$$
c(x_3,x_4,x_5)=x_5^3+c_4x_3x_4x_5.
$$
Now, if $c_4=0$, we get case (4). Otherwise we can scale coordinates so that $c_4=1$, and obtain case (5).
\item 
If $c_1=0$ and $c_2\ne 0$ or $c_3\ne 0$, then we can change $x_1$, $x_2$ and $x_5$ to get 
$$
c(x_3,x_4,x_5)=x_5^2(c_2x_3+c_3x_4),
$$
where at least one of $c_2$ or $c_3$ is not zero, since $X$ is not a cone. Then $X$ has a distinguished singularity in $\mathfrak{l}$, so this point must be fixed by $\Aut(X)$. Hence, this case is impossible. 
\item 
If $c_1=0$ and $c_2$ and $c_3=0$, then $c(x_3,x_4,x_5)=c_4x_3x_4x_5$, which means that $X$ is singular along a plane. Hence, this case is impossible.
\end{itemize}

Now we determine the automorphism groups of the two possible cases. 
In both cases, the group $\Aut(X)$ contains the infinite dihedral group $\bG_m\rtimes C_2$,
and $\bG_m$ acts on $\mathfrak{l}$ with the fixed points
$$
[1:0:0:0:0],\quad [0:1:0:0:0],
$$
which are swapped by the action of $C_2$. 
Observe that $X$ has $\sA_2\times\mathbb{A}^1$ singularity at every point of $\mathfrak{l}$ different from $[1:0:0:0:0]$ and $[0:1:0:0:0]$,
but $X$ has worse singularities at these two points, which implies that they form one $\Aut(X)$-orbit.
It follows that there exists an exact sequence 
$$
0\to H\to\Aut(X)\to \bG_m\rtimes C_2\to 0,
$$
where $H$ is the generic stabilizer of $\mathfrak l$. Similarly as in Proposition~\ref{prop:1-line}, a direct computation of $H$ completes the proof. 
\end{proof}

Consider the rational map $\chi\colon X\dasharrow\mathbb{P}^1_{x_1,x_2}\times\mathbb{P}^2_{x_3,x_4,x_5}$ given by
$$
\mathbf{x}\mapsto\big((x_1,x_2),(x_3,x_4,x_5)\big),
$$
for $X$ in Propositions~\ref{prop:1-line} and~\ref{prop:s}. 
 The description of automorphisms implies that 
$\chi$ is $\mathrm{Aut}(X)$-equivariant. Following the labelling in Propositions~\ref{prop:1-line} and \ref{prop:s},  $\chi$ is a birational map in all five cases except case (3). We discuss linearizations: 
\begin{itemize}
    \item[(1)] 
    $\Aut(X)=\fS_3$, acting on $\bP^1\times \bP^2$ diagonally via the usual action on $\bP^1$ and the permutation action on $\bP^2$. The action is linearizable, by Proposition~\ref{prop:25}. 
    \item[(2)]
    $\Aut(X)=\fS_4$, acting on $\bP^1\times \bP^2$ via the $\fS_3$-action on $\bP^1$ and the standard faithful linear action on $\bP^2$; this is linearizable, by Proposition~\ref{prop:25}. 
     \item[(3)]
     In this case, $\chi$ is not birational. Instead, the $\Aut(X)$-equivariant birational map $\rho: X\dashrightarrow\bP^3$ given by
     $$
(\mathbf x)\mapsto (x_2x_3,x_2x_4,x_2x_5,x_3^2+\zeta_3x_4^2+\zeta_3^2x_5^2)
$$
yields linearizability. In detail, $\rho$ factors through the birational maps $\pi$ and $\varphi$:
$$
\xymatrix{
X\ar@{-->}[r]^{\varphi}&X_{2,2}\ar@{-->}[r]^{\pi}& \bP^3,
}
$$
where $\varphi: X\dashrightarrow X_{2,2}\subset \bP^5$ is the unprojection from the plane $\{x_4=x_5=0\}$, given by          
     $$
(\mathbf x)\mapsto (x_1x_2,x_2^2,x_3x_2,x_4x_2,x_5x_2,x_3^2+\zeta_3x_4^2+\zeta_3^2x_5^2),
$$
and $X_{2,2}\subset\bP^5_{y_1,\ldots,y_6}$ is the intersection of two quadrics given by
$$
y_3^2+\zeta_3y_4^2 +\zeta_3^2y_5^2 + y_2y_6=
y_3^2 + \zeta_3^2y_4^2 + \zeta_3y_5^2 + y_1y_6=0.
$$ 
The singular locus of $X_{2,2}$ is the image of $\mathfrak l$, also a line. The birational map $\pi: X_{2,2}\dashrightarrow \bP^3$ is the
projection from this distinguished line, which fits the following commutative diagram:
$$
\xymatrix{
&\hat{X}_{2,2}\ar[ld]\ar[rd] &\\
X_{2,2}\ar@{-->}[rr]^{\pi} && \bP^3
}
$$
where $\hat{X}_{2,2}\to\bP^3$ is a blowup of $4$ general coplanar points, 
and $\hat{X}_{2,2}\to X_{2,2}$ is a birational map that contracts the strict transform of the plane spanned by these four points.      
    \item[(4)]
    The $\Aut(X)=\bG_m^2\rtimes C_2$-action on $\bP^2$ is generically free; any action of a finite subgroup is linearizable,  by Proposition~\ref{prop:25}. 
    \item[(5)] 
     Same as in Case (4), any action of a finite subgroup is linearizable.
\end{itemize}

\section{Conic}
\label{sect:conic}

We recall from \cite{yokocub} that the normal form of cubic threefolds $X$ singular along a conic $C$ is
\begin{multline*}
x_1q_1(x_4,x_5)+x_2q_2(x_4,x_5)+x_3q_3(x_4,x_5)+\\+q_4(x_1,x_2,x_3)l(x_4,x_5)+c(x_4,x_5)=0,
\end{multline*}
where $l$ is a linear form, $q_1,q_2,q_3,q_4$ are quadratic forms and $c$ is a cubic form. 
Let $\Pi=\{x_4=x_5=0\}$. Then $\Pi$ and $C$ are $\mathrm{Aut}(X)$-invariant and
$$
C=\{x_4=x_5=q_4=0\}\subset\Pi\subset X.
$$
We may assume that $C$ is smooth, otherwise the action either reduces to the case of a line, or is linearizable, via projection from the distinguished singularity. 
Similarly, we may assume that $\mathrm{Aut}(X)$ does not fix points in $C$. 

Changing coordinates on $\mathbb{P}^4$, we may assume that $q_4=x_1x_2+x_3^2$ and $l=x_5$. Then 
$$
\{x_5=0\}\cap X=2\Pi+\Pi^\prime,
$$
for some plane $\Pi^\prime\subset X$. Note that $2\Pi+\Pi^\prime$ is the only surface in the pencil of hyperplane sections containing $\Pi$ that splits as a union of three planes. Hence, the plane $\Pi^\prime$ is also $\mathrm{Aut}(X)$-invariant. Apriori, we have two possibilities:  $\Pi\ne \Pi^\prime$ (general case) and $\Pi=\Pi^\prime$ (special case). 

In the general case, the line $\Pi\cap\Pi^\prime$ intersects $C$ in two distinct points.
After another coordinate change, $\Pi^\prime=\{x_3=x_5=0\}$ and $X$ is given by
\begin{multline*}
x_1(a_1x_4x_5+a_2x_5^2)+x_2(b_1x_4x_5+b_2x_5^2)+x_3(x_4^2+c_1x_4x_5+c_2x_5^2)+\\
+(x_1x_2+x_3^2)x_5+e_1x_4^2x_5+e_2x_4x_5^2+e_3x_5^3=0,
\end{multline*}
for some $a_1,a_2,b_1,b_2,c_1,c_2,e_1,e_2,e_3\in k$.
Now, changing coordinates $x_1\mapsto x_1+\alpha x_3+\beta x_4$ and $x_2\mapsto x_2+\gamma x_3+\delta x_4$, for some $\alpha,\beta,\gamma,\delta\in k$,
we may further assume that $a_1=a_2=b_1=b_2=0$, and the defining equation of $X$ simplifies to
$$
x_3(x_4^2+c_1x_4x_5+c_2x_5^2)+(x_1x_2+x_3^2)x_5+e_1x_4^2x_5+e_2x_4x_5^2+e_3x_5^3=0.
$$
Finally, changing coordinates $x_3\mapsto x_3+\epsilon x_5$ and $x_4\mapsto x_4+\varepsilon x_5$, for some $\epsilon,\varepsilon\in k$,
we may assume that $c_1=c_2=0$, so that $X$ is given by
\begin{equation}
\label{eq:form-1}
x_3x_4^2+(x_1x_2+x_3^2)x_5+e_1x_4^2x_5+e_2x_4x_5^2+e_3x_5^3=0.
\end{equation}
We may assume that $(e_1,e_2,e_3)\ne(0,0,0)$ in \eqref{eq:form-1}, since otherwise $X$ would have additional singular point $[0:0:0:0:1]$,
which would be fixed by $\mathrm{Aut}(X)$. Moreover, scaling coordinates $x_1,x_2,x_3,x_4,x_5$ as 
$$
x_1\mapsto \frac{x_1}{s^2}, \quad  x_2\mapsto \frac{x_2}{s^2}, \quad x_3\mapsto \frac{x_3}{s^2}, \quad x_4\mapsto sx_4, \quad  x_5\mapsto s^4x_5,
$$
we scale $(e_1,e_2,e_3)$ as $(s^{6}e_1,s^{9}e_2,s^{12}e_3)$, so we really have 
$$
[e_1:e_2:e_3] \in \mathbb{P}(6,9,12)\simeq\mathbb{P}(1,3,2).
$$

Now, we turn to the special case when $\Pi=\Pi^\prime$.
Arguing as in the general case, we can change coordinates on $\mathbb{P}^4$ such that $X$ is given by 
\begin{equation}
\label{eq:form-2}
x_5(x_1x_2+x_3^2)+x_4^3+e_1x_4x_5^2+e_2x_5^3=0,
\end{equation}
for some $[e_1:e_2]\in\mathbb{P}(4,6)$. In this case, $X$ has an $\mathsf{A}_2$-singularity at a general point of the conic $C$. 

In both cases, consider the unprojection of the cubic $X$ from the plane $\Pi$, similar to the approach in \cite{CTZ-cubic,CMTZ}, for cubics with isolated singularities. Namely, if $X$ is given by \eqref{eq:form-1}, we can introduce a new coordinate 
$$
y_6=\frac{x_3x_4+e_1x_4x_5+e_2x_5^2}{x_5}=\frac{x_1x_2+x_3^2+e_3x_5^2}{-x_4},
$$
which gives an $\mathrm{Aut}(X)$-equivariant birational map $X\dasharrow X_{2,2}$, where $X_{2,2}$ is a complete intersection of two quadrics in $\mathbb{P}_{y_1,\ldots,y_6}^5$ given by 
$$
y_3y_4 + e_1y_4y_5 + e_2y_5^2 -y_5y_6=
y_1y_2 + y_3^2 + e_3y_5^2 + y_4y_6=0.
$$
In this case, $X_{2,2}$ has two isolated ordinary double points 
$$
[1:0:0:0:0:0],\quad [0:1:0:0:0:0],
$$
for general $[e_1:e_2:e_3]\in \mathbb{P}(6,9,12)$. If $X$ is given by \eqref{eq:form-2}, then we let
$$
y_6=\frac{x_4^2+e_1x_5^2}{x_5}=\frac{x_1x_2+x_3^2+e_3x_5^2}{-x_4},
$$
and obtain an $\mathrm{Aut}(X)$-equivariant birational map $X\dasharrow X_{2,2}$, where $X_{2,2}$ is given by 
$$
y_4^2 + e_1y_5^2 -y_5y_6=
y_1y_2 + y_3^2 + e_2y_5^2 + y_4y_6=0.
$$ 
In this case, the singular locus 
$$
\Sing(X_{2,2})=\{y_1y_2+y_3^2=y_4=y_5=y_6=0\}
$$ 
is also a smooth conic, for general $[e_1:e_2]\in\bP(4,6)$.
Let 
$$
\beta: \tilde{\bP}\to \bP^4
$$
be the blowup of $\Pi$. We have the commutative diagram

 \

\centerline{
\xymatrix{
 &  \tilde{\bP} \ar[ld]_{\beta}\ar[rd]^{\varpi} & \\
\bP^4 \ar@{-->}[rr] & & \bP^1_{x_4,x_5}
}
}

\ 

\noindent 
where $\varpi$ is a $\bP^3$-bundle and the dashed arrow is the projection from the plane $\Pi$. The restriction of $\varpi$ to $\tilde{X}$, the strict transform of $X$, is a quadric surface bundle.

%Another description of this geometry arises from uprojection of the plane  $\Pi$. 
%We obtain an intersection of two quadrics $X_{2,2}\subset \bP^5$, generically, with two ordinary double points, and a distinguished smooth point, the image of $\Pi$.   

%The following propositions are obtained by direction computation in {\tt magma}. 

%A priori, the conic $C$ could be singular. However, in this case, the $G$-action fixes a singular point on $X$, and the action is  

\begin{prop}\label{prop:lineauto-1}
    Let $X$ be a cubic given by \eqref{eq:form-1} 
    $$
x_3x_4^2+(x_1x_2+x_3^2)x_5+e_1x_4^2x_5+e_2x_4x_5^2+e_3x_5^3=0,
$$
    with parameters $e_1,e_2,e_3\in k$, such that $\Aut(X)$ does not fix any singular points of $X$. Then one of the following holds:
    \begin{enumerate}
    \item $e_1,e_2,e_3$ are general, and $\Aut(X)=\bG_m\rtimes C_2$, generated by 
    $$
  \tau_a: \mathrm{diag}(a,a^{-1},1,1,1),\,a\in k^\times, \quad \sigma_{(12)}:(\mathbf x)\mapsto (x_2,x_1,x_3,x_4,x_5).
    $$
        \item $e_1,e_3\ne 0$, $e_2=0$,  and $\Aut(X)=C_2\times(\bG_m\rtimes C_2)$, generated by $\tau_a,\sigma_{(12)}$ and 
        $$
        \eta_1:\mathrm{diag}(1,1,1,-1,1).
        $$
            \item $e_1=e_3=0$, $e_2\ne 0$,  and $\Aut(X)=C_3\times(\bG_m\rtimes C_2)$, generated by $\tau_a,\sigma_{(12)}$ and 
        $$
        \eta_2:\mathrm{diag}(1,1,1,\zeta_3,\zeta_3^2).
        $$
          \item $e_1=e_2=0$, $e_3\ne 0$, and $\Aut(X)=C_4\times(\bG_m\rtimes C_2)$, generated by $\tau_a,\sigma_{(12)}$ and 
        $$
        \eta_3:\mathrm{diag}(1,1,1,\zeta_4,-1).
        $$
          \end{enumerate}
\end{prop}
\begin{proof}
Note for general $e_1,e_2,e_3$, $X$ has $\sA_1\times\bA_1$-singularity at any point on the conic $C$ different from the two points
$$
[1:0:0:0:0],\quad [0:1:0:0:0]\in C,
$$
so the image of $G$ in $\Aut(C)=\PGL_2$ is contained in the infinite dihedral group $\bG_m\rtimes C_2$.
From the form of the equation, one sees that $\tau_a,\sigma\in\Aut(X)$ for all $a\in k^\times$ and $e_1,e_2,e_3\in k$, and thus generate the full group $\bG_m\rtimes C_2$. It follows that there is an exact sequence 
$$
0\to H\to \Aut(X)\to \bG_m\rtimes C_2\to 0,
$$
where $H$ is the generic stabilizer of $C$. Then a computation of $H$ based on the equation completes the proof. 
\end{proof}

\begin{prop}\label{prop:lineauto-2}
    Let $X$ be a cubic given by \eqref{eq:form-2} 
    $$
    x_5(x_1x_2+x_3^2)+x_4^3+e_1x_4x_5^2+e_2x_5^3=0,
    $$
    with parameters $e_1,e_2\in k$ such that $\Aut(X)$ does not fix any singular points of $X$. Then one of the following holds:
    \begin{enumerate}
    \item[(5)] $e_1,e_2\ne 0$, and $\Aut(X)=C_2\times\PGL_2$, where $\begin{pmatrix}
        a&b\\c&d
    \end{pmatrix}\in\PGL_2$ acts via 
    $$
    \begin{pmatrix}
        a^2&b^2&\zeta_4ab&&\\
        c^2&d^2&\zeta_4cd\\
        \frac{2ac}{\zeta_4}&   \frac{2bd}{\zeta_4}&ad+bc&&\\
        &&&ad-bc&\\
        &&&&ad-bc
    \end{pmatrix};
    $$  
    and $C_2$ acts via 
    $$
    \eta_1:\mathrm{diag}(1,1,1,-1,-1).
    $$
        \item[(6)] $e_1\ne 0$, $e_2=0$,  and $\Aut(X)=C_4\times\PGL_2$, generated by the $\PGL_2$-action described in case (5) and 
        $$
        \eta_2:\mathrm{diag}(1,1,1,\zeta_4,-\zeta_4).
        $$
          \item[(7)] $e_1=0, e_2\ne 0$, and $\Aut(X)=C_6\times\PGL_2$, generated by the $\PGL_2$-action described in case (5) and 
        $$
        \eta_3:\mathrm{diag}(1,1,1,-\zeta_3,-1)
        $$
    \end{enumerate}
\end{prop}
\begin{proof}
In this case, we observe that $\PGL_2\subset\Aut(X)$ with the generators given above.
   As before, one can directly compute the generic stabilizer $H$ of $C$ to obtain the three cases in the assertions. Note that $H$ always commutes with $\PGL_2$.
\end{proof}

We turn to the problem of linearization. 

\begin{theo} 
\label{thm:12}
    Let $X$ be the cubic given by \eqref{eq:form-1} or \eqref{eq:form-2} and $G=\mathfrak D_{n}$ the dihedral group generated by $\tau_a$ and $\sigma_{(12)}$ with $a=\zeta_{n}$, for some even $n$. Then the $G$-action on $X$ is not linearizable.  In particular, the $\Aut(X)$-action on $X$ is not linearizable. 
\end{theo}

\begin{proof}
The group $G$ contains the Klein four-group $\langle \tau_{-1}, \sigma_{(12)}\rangle$ where $\sigma_{12}$ fixes a cubic surface and the residual $C_2$ acting on $S$
fixes a smooth elliptic curve. The assertion then follows from \cite[Proposition 2.6]{CTZ-cubic}.
\end{proof}

\begin{rema}
Forms over non-closed fields of cubics given by \eqref{eq:form-2} have been considered in 
\cite[Section 11.3]{Colliot-P}. In the birational models as intersections of two quadrics in $
\bP^5$, the singular locus is a conic without rational points. The rationality of such threefolds remains an open problem. 

On the other hand, in the equivariant context, for the action of $G=C_2^2$, the corresponding conic has no $G$-fixed points; and the $G$-action is not linearizable, by Theorem~\ref{thm:12}. 
\end{rema}

\begin{theo}\label{thm:lineuni1}
    Let $X$ be a cubic given by \eqref{eq:form-1} or  \eqref{eq:form-2}. Then $X$ is $G$-unirational for all finite $G\subseteq\Aut(X)$.
\end{theo}
\begin{proof}
   The plane $\Pi\subset X$ (spanned by the conic $C$) is necessarily $G$-invariant. 
 The assertion then follows from Proposition~\ref{prop:unirat}.
\end{proof}

\section{The chordal cubic}\label{sect:chordal}
There is a unique cubic threefold
$X\subset \bP^4$ such that  $\Sing(X)$ is a rational normal quartic curve. It is known as the {\em chordal cubic}, and is given by 
\begin{align}\label{eqn:chordal}
    X=\{x_1x_4^2 + x_2^2x_5 - x_1x_3x_5 - 2x_2x_3x_4 + x_3^3=0\}\subset\bP^4_{x_1,\ldots,x_5}.
\end{align}
This $X$ is the secant variety of its singular locus $C:=\Sing(X)$. 
We have 
$$
\Aut(X)=\PGL_2,
$$ 
where $\PGL_2$ acts on $\bP^4$ via the usual embedding $C=\bP^1\hookrightarrow \bP^4$.
The linear system $|\cO_X(2)-C|$ gives rise to a $\PGL_2$-equivariant rational map $\rho: X\dashrightarrow S$,
where $S=\bP^2$ is the Veronese surface in $\bP^5$. This map fits into the following 
$\PGL_2$-equivariant commutative diagram:
$$
\xymatrix{
      & \tilde{X} \ar[dl]_{\beta} \ar[dr]^{\varpi}&         \\
X \ar@{-->}[rr]^{\rho} &  &  S}
$$
where $\beta$ is the blowup of $C$ in $X$ and $\varpi$ is a $\bP^1$-bundle. 

Let $G\subset \Aut(X)$ be a finite subgroup. Then $G$ is one of the following 
$$
C_n,\quad\fD_n,\quad \fA_4,\quad \fS_4,\quad \text{and}\quad\fA_5.
$$
%We may assume $G\ne C_n$ since $\Sing(X)^{C_n}\ne \emptyset$. 
For all such $G$, the induced action on $S$ is generically free, and linearizable. Moreover, the $G$-action on $X$ satisfies Condition {\bf (A)}, indeed, the Klein four-group 
$C_2^2\subset \PGL_2$, acting via
$$
\mathrm{diag}(1,-1,1,-1,1)\quad \text{and}\quad (\mathbf x)\mapsto (x_5,x_4,x_3,x_2,x_1),
$$
fixes the smooth point $[1:0:1:0:1]\in X$, and all other abelian groups are cyclic. 
Applying Proposition~\ref{prop:25}, we obtain linearizability of the $G$-action on $X$. To summarize, we obtain

\begin{theo}\label{theo:chordalmain}
    Let $X$ be the chordal cubic threefold given by \eqref{eqn:chordal}, and $G\subset \Aut(X)=\PGL_2$ a finite subgroup. Then the $G$-action on $X$ is linearizable. 
\end{theo}

\begin{rema}
We note that although all finite subgroups of $\Aut(X)$ in this case are linearizable, the $\Aut(X)$-action on $X$ is not linearizable. 
Indeed, from \cite[Section 2]{parkpoly}, we see that the $\bP^1$-bundle $\varpi: \tilde X\to\bP^2$ is the Schwarzenberger bundle $\mathcal S_3$, in the notation of \cite[Definition 1.2.7]{blancfanellibundle}. Precisely, it is the projectivization of the pullback of $\cO_{\bP^1\times\bP^1}(0,4)$ under the double cover $\bP^1\times\bP^1\to\bP^2$ ramified along a conic. The connected component of the identity $\Aut^\circ(\tilde X)=\PGL_2$. By \cite[Theorem F]{blancfanelliconnect}, there is no $\Aut^\circ(\tilde X)$-equivariant birational map to $\bP^3$.
\end{rema}

\section{Plane}
\label{sect:plane}

By \cite[Proposition 4.2]{yokocub}, a cubic threefold $X$ singular along the plane 
$$
\Pi=\{x_4=x_5=0\} \subset \bP^4
$$
is given by 
$$
x_1q_1(x_4,x_5)+x_2q_2(x_4,x_5)+x_3q_3(x_4,x_5)+c(x_4,x_5)=0,
$$
where $q_1,q_2,q_3$ are quadratic forms and $c$ is a cubic form in $x_4,x_5$. Viewing the first three terms as a quadratic form in $x_4$, $x_5$, over $k(x_1,x_2,x_3)$, we find that its discriminant defines a conic $C\subset \Pi=\bP^2_{x_1,x_2,x_3}$. If $C$ is non-reduced or singular then there is a distinguished point on $\Pi$ fixed by any group action, implying the linearizability. Thus, we may assume that $C$ is smooth.

\begin{prop}\label{prop:planeauto}
    Let $X$ be a cubic threefold singular along a plane. Assume that $\Aut(X)$ does not fix any singular points of $X$. Then, up to isomorphism, $X$ is given by 
    \begin{align}\label{eqn:sinplane}
x_1x_4^2+x_2x_4x_5+x_3x_5^2=0,
    \end{align}
    and $\Aut(X)$ is generated by elements 
    \begin{align}\label{eqn:GL2}
{\bf (x)} \mapsto 
{\bf (x)}\cdot\begin{pmatrix}
 d^2  &  -2bd     &  b^2  &     & \\
-cd & (ad+bc) & -ab &     & \\
 c^2 & -2ac     & a^2  &     & \\
      &         &      & a   & c \\
      &         &      & b   & d 
\end{pmatrix},\quad ad-bc\neq 0,
\end{align}
and 
    \begin{align}\label{eqn:stabi}
{\bf (x)} \mapsto 
{\bf (x)}\cdot\begin{pmatrix}
1 &  0    & 0  &     & \\
0 & 1 & 0 &     & \\
0& 0     & 1 &     & \\
    0  &   -\alpha      & -\beta    & \gamma   & 0 \\
    \alpha  & \beta        &    0  & 0   & \gamma 
\end{pmatrix},\quad \alpha,\beta\in k, \gamma\in k^\times.
\end{align}
%\begin{align}
%\mathbf{(x)}\mapsto (x_1+\alpha x_5,x_2-\alpha x_4+\beta %x_5,x_3-\beta x_4,\gamma x_4,\gamma x_5),
%\end{align}
\end{prop}
\begin{proof}
    Recall that $\Aut(X)$ preserves the conic $C\subset \Pi$. Up to isomorphism, we may assume that
$$
C=\{4x_1x_3-x_2^2=x_4=x_5=0\}.
$$
Since $X$ does not contain a distinguished singular point on $\Pi$, we may change variables to obtain the unique equation \eqref{eqn:sinplane} for $X$.

Since $\Aut(X)$ preserves $C$, its effective action on $\Pi$ is a subgroup of $\PGL_2$. We note that $\Aut(X)$ contains a subgroup isomorphic to 
$$
\GL_2=\left\{ \begin{pmatrix} a& c \\ b& d\end{pmatrix}, \quad ad-bc\neq 0\right\}, 
$$ 
generated by elements of the form \eqref{eqn:GL2}, which acts via $\PGL_2$ on $\Pi$. It follows that there is an exact sequence
$$
1\to H\to \Aut(X)\to \PGL_2\to 1,
$$
where $H$ is the generic stabilizer of $\Pi$. A direct computation shows that $H$ is generated by elements of the form \eqref{eqn:stabi}.

%$$
%\mathbf{(x)}\mapsto
%[d^2x_1-cdx_2+c^2x_3:-2bdx_1+(ad+bc)x_2-2acx_3:b^2x_1-x_2ab+a^2x_3:ax_4+bx_5: cx_4+dx_5],
%$$
%where $a,b,c,d\in k$ such that $ad-bc\ne 0$. 

\end{proof}

Proposition~\ref{prop:planeauto} implies that $\mathrm{Aut}(X)$ is isomorphic to the subgroup in $\GL_3$ consisting of $3\times 3$ matrices
$$
\left(
      \begin{array}{ccc}
        a & c &0\\
        b & d &0 \\
        \alpha& \beta & 1 \\
      \end{array}
    \right).
$$
This can also be explained as follows.
Let $f\colon\tilde{X}\to X$ be the blowup of $\Pi$, and let $E$ be the preimage of $\Pi$ in $\tilde{X}$. Then \begin{equation}
    \label{eqn:tildex}
\tilde{X}=\mathbb{P}^1_{y_1,y_2}\times\mathbb{P}^2_{z_1,z_2,z_3}.
\end{equation}
We may assume that 
$E$ is given by $z_3=0$, so that we can identify $E=\mathbb{P}^1_{y_1,y_2}\times\mathbb{P}^1_{z_1,z_2}$.
Then 
$$
\mathrm{Aut}(X)=\mathrm{Aut}(\tilde{X},E),
$$
where the latter group consists of automorphisms
\begin{multline*} 
(y_1,y_2)\times(z_1,z_2,z_3)\mapsto  \\(ay_1+by_2,cy_1+dy_2)\times(az_1+bz_2+\alpha z_3,cz_1+dz_2+\beta z_3,z_3).
\end{multline*}
The birational morphism $f$ is equivariant with respect to the described actions of $\mathrm{Aut}(\tilde{X},E)$, and can be explicitly presented as follows. Consider the Segre embedding  $\mathbb{P}^1_{y_1,y_2}\times\mathbb{P}^2_{z_1,z_2,z_3}\hookrightarrow\mathbb{P}^5$ given by
$$
(y_1,y_2)\times(z_1,z_2,z_3)\mapsto (y_1z_1,y_1z_2,y_1z_3,y_2z_1,y_2z_2,y_2z_3).
$$
Composing it with the linear projection $\mathbb{P}^5_{t_1,t_2,t_3,t_4,t_5,t_6}\dasharrow\mathbb{P}^4$ 
given by 
\begin{align}\label{eqn:projpot}
    (\mathbf t)\mapsto (t_1,-t_2-t_4,t_5,t_6,t_3),
\end{align}
we obtain the map $f$; the resulting image of $\bP^1\times\bP^2$ in $\bP^4$ is the cubic $X$ given by \eqref{eqn:sinplane}. 
The map \eqref{eqn:projpot} is a projection from the point 
$$
[0:1:0:-1:0:0]\in \bP^5.
$$
This point is not in the image of $\bP^1\times\bP^2$, but it is fixed by the image of $\Aut(\tilde X,E)$ in $\PGL_6$. Note that $f$ induces a double cover $E\to \Pi$ which is ramified in $C$, which explains why this conic is $\mathrm{Aut}(X)$-invariant.

\begin{theo}\label{thm:plane}
Let $X$ be the cubic threefold given by \eqref{eqn:sinplane} and $G\subset \Aut(X)$ a finite group. 
Then the $G$-action on $X$ is linearizable. 

\end{theo}

\begin{proof}
We see that the $\Aut(X)$-equivariant birational model $\tilde{X}$ of $X$ is isomorphic to $\bP^1\times \bP^2$, 
with coordinates as in \eqref{eqn:tildex}. 
Any finite subgroup $G\subset \Aut(X)$ acts linearly and generically freely on $\bP^2_{z_1,z_2,x_3}$, and lifts to the vector bundle 
$$
\bA^2_{y_1,y_2}\times \bP^2_{z_1,z_2,x_3} \to \bP^2_{z_1,z_2,x_3}.
$$
By Proposition~\ref{prop:25}, the $G$-action is linearizable.      
\end{proof}

\bibliographystyle{plain}
\bibliography{sing-cube}

\begin{thebibliography}{10}

\bibitem{All}
D.~Allcock.
\newblock The moduli space of cubic threefolds.
\newblock {\em J. Algebr. Geom.}, 12(2):201--223, 2003.

\bibitem{blancfanellibundle}
J.~Blanc, A.~Fanelli, and R.~Terpereau.
\newblock Automorphisms of {$\Bbb P^1$}-bundles over rational surfaces.
\newblock {\em \'Epijournal G\'eom. Alg\'ebrique}, 6:Art. 23, 47, 2022.

\bibitem{blancfanelliconnect}
J.~Blanc, A.~Fanelli, and R.~Terpereau.
\newblock Connected algebraic groups acting on three-dimensional {M}ori
  fibrations.
\newblock {\em Int. Math. Res. Not. IMRN}, (2):1572--1689, 2023.

\bibitem{CGH}
S.~Casalaina-Martin, S.~Grushevsky, K.~Hulek, and R.~Laza.
\newblock Complete moduli of cubic threefolds and their intermediate
  {Jacobians}.
\newblock {\em Proc. Lond. Math. Soc. (3)}, 122(2):259--316, 2021.

\bibitem{CMTZ}
I.~Cheltsov, L.~Marquand, Yu. Tschinkel, and Zh. Zhang.
\newblock Equivariant geometry of singular cubic threefolds {I}{I}, 2024.
\newblock {\tt arXiv:2405.02744}.

\bibitem{CTZ-cubic}
I.~Cheltsov, Yu. Tschinkel, and Zh. Zhang.
\newblock Equivariant geometry of singular cubic threefolds.
\newblock {\em Forum Math. Sigma}, 13:Paper No. e9, 2025.

\bibitem{CTZ-burk}
I.~Cheltsov, Yu. Tschinkel, and Zh. Zhang.
\newblock Equivariant geometry of the {S}egre cubic and the {B}urkhardt
  quartic.
\newblock {\em Selecta Math. (N.S.)}, 31(1):Paper No. 7, 36, 2025.

\bibitem{CTZ-uni}
I.~Cheltsov, Yu. Tschinkel, and Zh. Zhang.
\newblock Equivariant unirationality of {F}ano threefolds, 2025.
\newblock {\tt arXiv:2502.19598}.

\bibitem{Colliot-P}
J.-L. Colliot-Th\'{e}l\`ene and A.~Pirutka.
\newblock Certaines fibrations en surfaces quadriques r{\'e}elles, 2024.
\newblock {\tt arXiv:2406.00463}.

\bibitem{CTS}
J.-L. Colliot-Th\'{e}l\`ene and P.~Salberger.
\newblock Arithmetic on some singular cubic hypersurfaces.
\newblock {\em Proc. London Math. Soc. (3)}, 58(3):519--549, 1989.

\bibitem{coray}
D.~F. Coray and M.~A. Tsfasman.
\newblock Arithmetic on singular {D}el {P}ezzo surfaces.
\newblock {\em Proc. London Math. Soc. (3)}, 57(1):25--87, 1988.

\bibitem{DR}
A.~Duncan and Z.~Reichstein.
\newblock Versality of algebraic group actions and rational points on twisted
  varieties.
\newblock {\em J. Algebr. Geom.}, 24(3):499--530, 2015.

\bibitem{Kang}
M.~Hajja and M.~Ch. Kang.
\newblock Some actions of symmetric groups.
\newblock {\em J. Algebra}, 177(2):511--535, 1995.

\bibitem{parkpoly}
I.-K. Kim, J.~Park, and J.~Won.
\newblock K-polystability of the {F}irst {S}ecant {V}arieties of {R}ational
  {N}ormal {C}urves.
\newblock {\em Int. Math. Res. Not. IMRN}, (7):rnaf088, 2025.

\bibitem{kollarcubic}
J.~Koll\'ar.
\newblock Unirationality of cubic hypersurfaces.
\newblock {\em J. Inst. Math. Jussieu}, 1(3):467--476, 2002.

\bibitem{LiuXu}
Y.~Liu and C.~Xu.
\newblock K-stability of cubic threefolds.
\newblock {\em Duke Math. J.}, 168(11):2029--2073, 2019.

\bibitem{pinardin}
A.~Pinardin, A.~Sarikyan, and E.~Yasinsky.
\newblock Linearization problem for finite subgroups of the plane {C}remona
  group, 2024.
\newblock {\tt arXiv:2412.12022}.

\bibitem{viktorova}
S.~Viktorova.
\newblock On the classification of singular cubic threefolds, 2025.
\newblock To appear: {\it Trans. Am. Math. Soc,} {\tt arXiv:2304.10452}.

\bibitem{yokocub}
M.~Yokoyama.
\newblock Stability of cubic 3-folds.
\newblock {\em Tokyo J. Math.}, 25(1):85--105, 2002.

\end{thebibliography}

\end{document}